\documentclass[12pt,reqno]{amsart}

\usepackage{amsmath,amsthm,amssymb,fullpage}

\newcommand*{\midcup}{\mathbin{\scalebox{1.2}{\ensuremath{\bigcup}}}}%

\usepackage[T1]{fontenc}
\usepackage{latexsym}
\usepackage[dvips]{graphics}
\usepackage{epsfig}
\usepackage{amsmath,amsfonts,amsthm,amssymb,amscd}
\input amssym.def
\input amssym.tex
\usepackage{color}
\usepackage{hyperref}
\usepackage{url}
\usepackage{array}
\usepackage{caption}
\usepackage{booktabs}

\newcommand{\burl}[1]{\textcolor{blue}{\url{#1}}}

\numberwithin{equation}{section}

\newtheorem{thm}{Theorem}[section]

\newtheorem{que}[thm]{Question}

\theoremstyle{plain}

\newtheorem{corollary}[thm]{Corollary}

\newtheorem{definition}[thm]{Definition}

\newtheorem{lemma}[thm]{Lemma}

\newtheorem{theorem}[thm]{Theorem}

\newtheorem{remark}[thm]{Remark}

\newcommand\be{\begin{equation}}

\newcommand\ee{\end{equation}}

\newcommand\bea{\begin{eqnarray}}

\newcommand\eea{\end{eqnarray}}

\newcommand\bi{\begin{itemize}}

\newcommand\ei{\end{itemize}}

\newcommand\ben{\begin{enumerate}}

\newcommand\een{\end{enumerate}}

\newcommand\bc{\begin{center}}

\newcommand\ec{\end{center}}

\newcommand\ba{\begin{array}}

\newcommand\ea{\end{array}}

\newcommand{\R}{\ensuremath{\mathbb{R}}}

\newcommand{\Z}{\ensuremath{\mathbb{Z}}}

\newcommand{\Q}{\mathbb{Q}}

\newcommand{\N}{\mathbb{N}}

\newcommand\frakfamily{\usefont{U}{yfrak}{m}{n}}

\DeclareTextFontCommand{\textfrak}{\frakfamily}

\newcommand{\clopin}[2] { (M/#1,M/#2 ] \cup}

\newcommand{\hr}[1]{\href{#1}{\url{#1}}}

\title{Geometric-Progression-Free Sets over Quadratic Number Fields}

\author{Andrew Best}
\address{Department of Mathematics, Ohio State University, Columbus, OH 43210}
\email{\textcolor{blue}{\href{mailto:best.221@osu.edu)
}{best.221@osu.edu}}}

\author{Karen Huan}
\address{Department of Mathematics and Statistics, Williams College, Williamstown, MA 01267}
\email{\textcolor{blue}{\href{mailto:klh1@williams.edu}{klh1@williams.edu}}}

\author{Nathan McNew}
\address{Department of Mathematics, Towson University, Towson, MD 21252}
\email{\textcolor{blue}{\href{mailto:nmcnew@towson.edu}{nmcnew@towson.edu}}}

\author{Steven J. Miller}
\address{Department of Mathematics and Statistics, Williams College, Williamstown, MA 01267}
\email{\href{mailto:sjm1@williams.edu}{\textcolor{blue}{sjm1@williams.edu}}, \href{mailto:Steven.Miller.MC.96@aya.yale.edu}{\textcolor{blue}{Steven.Miller.MC.96@aya.yale.edu}}}

\author{Jasmine Powell}
\address{Department of Mathematics, University of Michigan, Ann Arbor, MI 48104}
\email{\textcolor{blue}{\href{mailto:jtpowell@umich.edu}{jtpowell@umich.edu}}}

\author{Kimsy Tor}
\address{Department of Mathematics, Universit\'{e} Pierre et Marie Curie, Paris 75005, France}
\email{\textcolor{blue}{\href{mailto:kimsy.tor@etu.upmc.fr
}{kimsy.tor@etu.upmc.fr
}}}

\author{Madeleine Weinstein}
\address{Department of Mathematics, Harvey Mudd College, Claremont, CA 91711 }
\email{\textcolor{blue}{\href{mailto:mweinstein@hmc.edu}{mweinstein@hmc.edu}}}

\thanks{This research was conducted as part of the 2014 SMALL REU program at Williams College and was supported by NSF grant DMS1347804 and DMS1265673 and Williams College. We would also like to thank the participants of the SMALL REU for many helpful discussions. }



\date{\today}

\begin{document}
\begin{abstract}

In Ramsey theory one wishes to know how large a collection of objects can be while avoiding a particular substructure. A problem of recent interest has been to study how large subsets of the natural numbers can be while avoiding 3-term geometric progressions. Building on recent progress on this problem, we consider the analogous problem over quadratic number fields.

We first construct high-density subsets of the algebraic integers of an imaginary quadratic number field that avoid 3-term geometric progressions. When unique factorization fails or over a real quadratic number field, we instead look at subsets of ideals of the ring of integers. Our approach here is to construct sets ``greedily,'' a generalization of the greedy set of rational integers considered by Rankin.  We then describe the densities of these sets in terms of values of the Dedekind zeta function.

Next, we consider geometric-progression-free sets with large upper density. We generalize an argument by Riddell to obtain upper bounds for the upper density of geometric-progression-free subsets, and construct sets avoiding geometric progressions with high upper density to obtain lower bounds for the supremum of the upper density of all such subsets. Both arguments depend critically on the elements with small norm in the ring of integers.

\end{abstract}

\maketitle

\maketitle

\vspace{-1em}

\tableofcontents

\section{Introduction}
In 1961, Rankin \cite{Rankin} introduced the problem of constructing large subsets of the integers which do not contain geometric progressions. A set is free of geometric progressions if it does not contain any three terms in a progression of the form $\{n, nr, nr^2\}$ and $r \in \N \setminus \{ 1 \}$. A simple example of such a set is the squarefree integers, which have density $6/\pi^2 \approx 0.60793$. Rankin constructed a larger set, with asymptotic density approximately $0.71974$.  Modifying the construction of this set slightly, (see \cite{McNew}) one can get a small improvement, producing a set with density about 0.72195.   It is not known what the greatest possible asymptotic density of such a set is, however several authors have considered sets with greater upper asymptotic density.  Riddell \cite{Riddell}, Brown and Gordon \cite{BroGor}, Beiglb\"ock, Bergelson, Hindman, and Strauss \cite{BBHS},  Nathanson and O'Bryant \cite{NatO'Bry1} and most recently the third named author, \cite{McNew}, have improved the bounds on the greatest possible upper density of such a set. It is now known that the greatest possible upper density of such a set lies between $0.81841$ and $0.81922$.  Other papers have also considered progressions which are allowed to have rational ratios, however that will not be pursued further here.

In this paper, we generalize the problem to quadratic number fields. We study the geometric-progression-free subsets of the algebraic integers (or ideals) and give bounds on the possible densities of such sets. In an imaginary quadratic field with unique factorization we are able to consider the possible densities of sets of algebraic integers which avoid 3-term geometric progressions, and so we have chosen to pay special attention to those fields. Most of our results also apply to the ideals of any quadratic number field, however, which we also state when possible. We begin by characterizing the  ``greedy'' geometric-progression-free sets in each case, and then give bounds for the greatest possible upper density of such a set.  We conclude our paper with some numerical data and several open questions. \\

\section{Rankin's Greedy Set}

Let $G_3^*=\{1, 2, 3, 5, 6, 7, 8, 10, 11, 13, 14, 15, 16, 17, 19, 21, 22, 23, \dots\}$ (see \href{https://oeis.org/A000452}{OEIS A000452} \cite{OEIS}) be the set of positive integers obtained greedily by including those integers which do not introduce a 3-term geometric progression with the integers already included in the set. Rankin \cite{Rankin} characterized $G_3^*$ as the set of those integers with prime exponents coming only from the set, $A_3^*=\{0, 1, 3, 4, 9, 10, 12, 13,\ldots\}$ (see \href{https://oeis.org/A005836}{OEIS A005836}  \cite{OEIS}), of nonnegative integers obtained by greedily including integers which avoid 3-term arithmetic progressions. The set $A_3^*$ consists of those integers which do not contain a $2$ in their base $3$ expansion.

This characterization shows that the inclusion of any given integer in the set $G_3^*$ depends solely on the powers of the primes in its prime factorization. Thus, to find the density of $G_3^*$, one needs only consider the density of those integers which contain only primes raised to acceptable powers in their prime factorization. Because the acceptable powers are the integers in $A_3^*$, we have, for a fixed prime $p$, that the density of the positive integers whose factorization contains an acceptable power of $p$ is given by \be\label{eq:densitywitharith} \left(\frac{p-1}{p}\right)\sum_{i \in A_3^*} \frac{1}{p^i}\ =\ \left(\frac{p-1}{p}\right)\prod_{i \geq 0} \left(1+\frac{1}{p^{3^i}}\right).
\ee
By the Chinese remainder theorem, we can then compute the asymptotic density of $G_3^*$ as an Euler product over all primes:
\bea\label{rankinset}
d(G_3^*) &\ = \ & \displaystyle\prod_{p}\left(\frac{p-1}{p}\right)\displaystyle\prod_{i=0}^{\infty}\left(1+\frac{1}{p^{3^i}}\right)\nonumber\\
&\ = \ &\displaystyle\prod_{p}\left(1-\frac{1}{p^2}\right)\displaystyle\prod_{i=1}^{\infty}\frac{1-\frac{1}{p^{2\cdot 3^i}}}{1-\frac{1}{p^{3^i}}}\nonumber\\
&\ = \ &\frac{1}{\zeta(2)}\displaystyle\prod_{i=1}^{\infty}\frac{\zeta(3^i)}{\zeta(2\cdot3^i)} \ \approx \ 0.71974.
\eea

Thus the supremum of the densities of all geometric-progression-free subsets of the natural numbers is bounded below by $0.71974$.


\section{Greedy Set over Quadratic Number Fields}
\label{greedysection}

We begin by generalizing Rankin's calculation of the density of the greedy set of integers to the algebraic integers $\mathcal{O}_K$ over a quadratic number field, $K=Q(\sqrt{d})$, where $d$ is a squarefree integer.  (For a review of the algebraic number theory used here, the reader is referred to \cite{Marcus}.) If $d<0$, $K$ is an imaginary quadratic number field with finitely many units, so there are finitely many algebraic integers with bounded norm. It is thus possible in this case to define the density of a subset of the algebraic integers.
\begin{definition}
Let $K$ be an imaginary quadratic number field. The \textbf{density} $d(A)$ of a set $A \subset \mathcal{O}_K$ is
\bea \label{den1}
d(A) \ = \  \lim_{n \to \infty} \frac{|A \cap I_n|}{|I_n|},
\eea
where $I_n=\{m \in \mathcal{O}_K : N(m) \leq n\} \text{ and } N(m) \text{ is the norm of m}$.
\end{definition}

Most of our results require unique factorization and so, when studying subsets of algebraic integers of an imaginary quadratic number field which avoid geometric progressions, we must restrict our attention to those imaginary quadratic number fields which have class number 1, of which there are only nine ($d = -1,-2,-3,-7,-11,-19,-43,-67,-163$).  When either unique factorization fails or when $K$ is a real quadratic number field we can instead consider subsets of the ideals of $\mathcal{O}_K$ which avoid 3-term geometric progressions and obtain similar results about the densities of the ideals in this case which apply to all quadratic number fields.  The density of a subset of the ideals of $\mathcal{O}_K$ is defined similarly.
\begin{definition}
Let $K$ be a number field. The density $d(A)$ of a set of ideals $A$ of $\mathcal{O}_K$ is defined to be
\bea \label{den2}
d(A)\ =\ \lim_{n \to \infty} \frac{|A \cap I_n|}{|I_n|},\eea
where $I_n = \{\mathfrak{I} \subseteq \mathcal{O}_K : N(\mathfrak{I}) \ = \  |\mathcal{O}_K / \mathfrak{I}|  \leq n\}$.
\end{definition}

Before stating our results about the possible densities of sets avoiding geometric progressions, we make a note about progressions with unit ratios. Throughout this paper, we restrict our study of geometric progressions to those progressions, $\{n,nr,nr^2\}$, where the ratio, $r$, is not a unit in the number field.  Over the integers this distinction is made by insisting that every term in a 3-term geometric progression be unique.  However, that is not necessarily the case here, since one can find progressions, such as $\{2,2i,-2\}$ over the Gaussian integers with ratio $i$, which have unit ratio and distinct terms.  Nevertheless, our decision is motivated by the additional structure that progressions with non-unit ratios possess, the fact that each such progression is naturally ordered by starting with the element with least norm, and, most importantly, and the consequence that this choice yields results directly comparable to the ideal case discussed in greater generality below.

Also, note that the norm doesn't totally order the algebraic integers (or ideals) in a number field, and hence the notion of a greedy set is not necessarily well defined, as it is over the integers.  However, as we will see, inclusion in the greedy set avoiding geometric progressions depends only on the norm, so all elements with the same norm are either included or excluded. Thus the order in which the elements (integers or ideals) of same norm are taken is irrelevant. 

Suppose first that $K$ is an imaginary quadratic number field with class number 1. Then the density of the set, $G_{K,3} \subset \mathcal{O}_K$ which greedily avoids 3-term geometric progressions with non-unit ratios $r \in \mathcal{O}_K$ can be found as follows:

\begin{theorem}\label{thm:densityfieldintegerrat}
Let $K$ be an imaginary quadratic number field with class number 1, and let $f : \N \to \R $ be defined by
\bea \label{fun}
f(x) \ =\  \left(1-\frac{1}{x}\right) \prod_{i=0}^{\infty}\left(1+\frac{1}{x^{3^i}}\right).
\eea
Then the density of the greedy set, $G_{K,3}$, of algebraic integers which avoid $3$-term progressions with non-unit ratios in $\mathcal{O}_K$ is given by
\be \label{eq:densitygk3}
d(G_{K,3})\  =\  \left(\prod_{p\ \text{{\rm inert}}} f(p^2)\right) \left(\prod_{p\ \text{{\rm splits}}} f(p)^2\right)  \left(\prod_{p\ \text{ {\rm ramifies}}} f(p)\right),
\ee
where the product is taken over rational primes $p$ depending on whether $p$ splits, ramifies or remains inert in $K$.
\end{theorem}

\begin{proof}
The method of proof is similar to that of Rankin's result. If $\{n,nr, nr^2\} \subset \mathcal{O}_K$ forms a geometric progression with $r \in \mathcal{O}_K$ non-zero and a non-unit, then any prime, $p$, of $\mathcal{O}_K$ which divides $n$ exactly $a$ times and $r$ exactly $b$ times will appear to the powers $a,a+b$ and $a+2b$ in the terms of the progression respectively.  In particular, these exponents will form an arithmetic progression. So, greedily avoiding geometric progressions amounts to greedily avoiding arithmetic progressions among the exponents of the primes.  Thus, as in Rankin's set over the integers, an algebraic integer is included in the greedy set, $G_{K,3}$, if and only if the exponents on the primes in its prime factorization come from $A_3^*$.

Now, for a fixed prime $q$ of $\mathcal{O}_K$ we determine, analogously to equation \eqref{eq:densitywitharith} over the integers, that the density of the algebraic integers which have an acceptable power of $q$ in their prime factorization is
\be\label{Okprimefact}
\left(\frac{N(q)-1}{N(q)}\right)\sum_{i \in A_3^*} \frac{1}{N(q)^i}\ =\ f(N(q)).
\ee

So, by the Chinese remainder theorem, the density is now given by a product of factors of $f(x)$, as described above. This product can be converted from a product over the primes in $K$ to a product over rational primes using the fact that each prime of $K$ lies above a rational prime. The contribution of each rational prime will depend upon whether it remains inert, splits, or ramifies in $K$.

If a rational prime, $p$, remains inert in $K$, then the factor representing this prime is the same as in Rankin's product, with one modification: Instead of  $p$ residues mod $p$, there are now $N(p)=p^2$ residues modulo $p$ in $\mathcal{O}_K$. Thus this prime contributes a factor of $f(p^2)$ rather than $f(p)$.

If a rational prime, $p$, splits, it splits into two conjugate primes of norm $p$, each of which can be raised to different powers and for each we have a factor of $f(p)$. Multiplying these together, the contribution from this kind of rational prime is $f(p)^2$.

Finally, if a prime, $p$, ramifies, it corresponds to exactly one prime of $K$ of norm $p$. Thus we account for divisibility by powers of this prime with the factor $f(p)$.
\end{proof}

Then, as in Rankin's argument, the density of this  greedy set serves as a lower bound for the greatest possible density of a set free of this kind of geometric progression. Approximations of the resulting lower bounds in each of the 9 class number 1 imaginary quadratic number fields are given in Table \ref{tab:greedysetdens}.
\begin{table}[h]
\centering
\begin{tabular}{|c|c|}
\hline
 $d$ & Density of the greedy set\\
\hline
-1 & 0.762340 \\
-2 & 0.693857\\
-3 & 0.825534\\
-7 & 0.674713\\
-11 & 0.742670\\
-19 & 0.823728 \\
-43 & 0.898250 \\
-67 & 0.917371 \\
-163 & 0.933580\\
\hline
\end{tabular}
\vspace{.1em}
\caption{Density of the greedy set, $G_{K,3}\subset \mathcal{O}_K$, of algebraic integers which avoid 3-term geometric progressions with ratios in $\mathcal{O}_K$ for each of the class number $1$ imaginary quadratic number fields $K=\Q\left(\sqrt{d}\right)$.}
\label{tab:greedysetdens}
\end{table}

For the general case of quadratic number fields, we use the Dedekind zeta function in place of the Riemann zeta function. It is defined as follows.

\begin{definition}
Let $K$ be an algebraic number field. Then the Dedekind zeta function for $K$ is defined as \bea \label{ded}
\zeta_K(s)\ = \ \sum_{\mathfrak{I} \subset \mathcal{O}_K} \frac{1}{(N_{K/Q}(\mathfrak{I}))^s}, \eea
where $\mathfrak{I}$ ranges through the non-zero ideals of the ring of integers $\mathcal{O}_K$ of $K$ and $N_{K/Q}(\mathfrak{I})$ is the absolute norm of $\mathfrak{I}$.
\end{definition}
The ring of integers of a quadratic number field is a  Dedekind domain so the ideals factor uniquely into prime ideals. Using this unique factorization, the Dedekind zeta function can be expressed as an Euler product over all the prime ideals $\mathfrak{P}$ of $\mathcal{O}_K$, that is
\be \label{eq:dedekindzetaproduct} \zeta_K(s)\ =\ \prod_{\mathfrak{P} \subset \mathcal{O}_K} \frac{1}{1-(N_{K/Q}(\mathfrak{P}))^{-s}}
\ee
for $Re(s)>1.$

In order to state the result, we will also need the quadratic character of a number field, which is defined as follows.
\begin{definition}
Let $K$ be a quadratic field with discriminant $D_K$.
The \textbf{quadratic character} of $K$ is $\chi_K:\mathbb{Z}^+ \longrightarrow \mathbb{C}, \chi_K(m)=(\frac{D_K}{m})$ where $(\frac{D_K}{m})$ is the Jacobi symbol.
\end{definition}

With these definitions in place, the general case of the density of the greedy set of ideals avoiding geometric progressions where the ratios are ideals can now be stated.

\begin{theorem}\label{greedyidealdensity}
Let $K$ be a quadratic number field and let $f : \N \to \R $ be defined by
\bea \label{fun2}
f(x) \ =\  \left(1-\frac{1}{x}\right) \prod_{i=0}^{\infty}\left(1+\frac{1}{x^{3^i}}\right).
\eea
Then the density of the greedy set of ideals, $G^*_{K,3}$, of $\mathcal{O}_K$ that avoid 3-term geometric progressions with ratio a (non-trivial) ideal of $\mathcal{O}_K$ is
\be \label{eq:greedydensity} d(G^*_{K,3}) \ =\   \frac{1}{\zeta_K(2)}\displaystyle\prod_{i=1}^{\infty}\frac{\zeta_K(3^i)}{\zeta_K(2\cdot3^i)}.
\ee
\end{theorem}

\begin{proof}
The method of proof is essentially the same as that of Theorem \ref{thm:densityfieldintegerrat}. As before, an ideal is in $G^*_{K,3}$ if and only if the exponents on the primes ideals in its prime factorization appear only with exponents from $A_3^*$. For a fixed prime ideal $\mathfrak{P}$ of $\mathcal{O}_K$, the density of the ideals divisible by $\mathfrak{P}$ is $1/N(\mathfrak{P})$, and likewise, the density of ideals exactly divisible by an acceptable power of $\mathfrak{P}$ is given by $f(N(\mathfrak{P}))$. Thus, the density is again given by a product of factors of $f(x)$ as described above.  We can likewise write this product as a product over rational primes depending on whether the rational primes remain inert, split, or ramify in $K$, with the norms behaving exactly as they did in the case of class number 1 imaginary quadratic number fields, so the product has the same structure as that of Theorem \ref{thm:densityfieldintegerrat}. We now characterize the behavior of the rational primes using the quadratic character. In particular, a prime ideal with $\chi_K(p)=-1$ remains inert, an ideal with $\chi_K(p)=1$ splits, and an ideal with $\chi_K(p)=0$ ramifies. So we get that
\[d(G^*_{K,3}) \ =\ \left(\prod_{\chi_K(p)=-1} f(p^2)\right) \left(\prod_{\chi_K(p)=1} f(p)^2\right) \left(\prod_{ \chi_K(p)=0} f(p)\right). \] The Euler product \eqref{eq:dedekindzetaproduct} and the steps used to derive equation \eqref{rankinset} can be used to transform this product into the product 
\[ d(G^*_{K,3}) \ =\   \frac{1}{\zeta_K(2)}\displaystyle\prod_{i=1}^{\infty}\frac{\zeta_K(3^i)}{\zeta_K(2\cdot3^i)} \]
in terms of values of the Dedekind zeta function.
\end{proof}

\subsection{Two Universal Bounds}
We can use Theorem \ref{greedyidealdensity} to give both upper and lower bounds for the density of the greedy set of ideals avoiding 3-term geometric progressions in any quadratic number field. First, an observation:
\begin{remark}
Define $f(x)$ as in Theorem \ref{greedyidealdensity}. Then, for all primes $p$, we have
\bea \label{ineq}
0 \ < \ f(p)^2 \ < \ f(p) \ < \ f(p^2) \ < \ 1.
\eea
\end{remark}
\noindent With that said, we may proceed.
\begin{corollary} \label{cor:universalbounds}
Let $K$ be a quadratic number field, and $G_{K,3}^*$ the greedy set of ideals of $\mathcal{O}_K$ that avoid 3-term geometric progressions. Then
\be \label{eq:universalbounds} \prod_{p} f(p)^2 \ < \ d(G_{K,3}^*) \ < \  \prod_{p} f(p^2) \ee
and approximating these products we find that
\be\label{den}
 0.518033\  < \ d(G_{K,3}^*) \ <\  0.939735 .
\ee

\end{corollary}
\begin{proof}
We use the remark above to give upper and lower bounds for the product in equation \eqref{eq:greedydensity}.  By Theorem \ref{greedyidealdensity}, we have
\begin{equation}
\begin{split} \label{dK3upperbound}
d(G_{K,3}^*) \ &= \ \left(\prod_{\chi_K(p)=-1} f(p^2)\right) \left(\prod_{\chi_K(p)=1} f(p)^2\right) \left(\prod_{ \chi_K(p)=0} f(p)\right) \\
\ &> \ \prod_{p} f(p)^2 \ > \ 0.518033,
\end{split}
\end{equation}

and likewise
\be\label{densityGk3bound}
d(G^*_{K,3})\ <\  \prod_{p} f(p^2)\  < \ 0.939735.
\ee
\end{proof}

Note that the inequalities of Corollary \ref{cor:universalbounds} are strict because it is neither possible for all of the rational primes to split in a quadratic number field nor for all to remain inert.  Interestingly, however, we find that the upper bound of inequality \eqref{eq:universalbounds} is achieved by a different greedily constructed set: the greedy set of algebraic integers (or ideals) which avoid progressions where the ratio between consecutive terms is a rational integer, which we will see in Theorem \ref{thm:densityintegralrat}.

\subsection{Greedy Set Avoiding Geometric Progressions with Ratios in \texorpdfstring{$\mathbb{Z}$}{Z}}

We consider first the Gaussian integers, $\Z[i]$, which are the algebraic integers of the quadratic number field $\Q\left(\sqrt{-1}\right)$, and look at the subset, $H^{*}_{\Q(i),3}$, chosen greedily to avoid geometric progressions with common ratios contained in
$\Z \setminus \{ \pm 1\}$. For example, the Gaussian integers $3 + i, {6 + 2i}, 12 + 4i$ form a geometric progression in the Gaussian integers with common ratio 2. We will calculate below the density of this greedy set.
Consider the Gaussian integers as lattice points in the complex plane, and note that any geometric progression of Gaussian integers with ratios in $\Z \setminus \{ \pm 1\}$ will lie on a line through the origin with rational slope.  Scaling a Gaussian integer by a rational integer twice produces three collinear points. We then characterize the greedy set $H^{*}_{\Q(i),3}$ as follows:
\begin{lemma}\label{lem:greedyint}
A Gaussian integer $a + bi$ is excluded from $H^{*}_{\Q(i),3}$ exactly when it can be written in the form $a + bi = k(c + di)$ where $k$ is an element excluded from $G_3^*$ and $(c, d)= 1$.
\end{lemma}

\begin{proof}
This follows from the definition of $G_3^*$  since it greedily avoids rational integral ratios.
\end{proof}

\noindent Having characterized this greedy set, we can now compute its density.

\begin{theorem}\label{thm:densityintegralrat}
The density of $H^{*}_{\Q(i),3}$, the greedy set of Gaussian integers that avoid geometric progressions with rational integral ratios, is equal to $\prod_{p}f(p^2)$, approximately $0.939735$.
\end{theorem}
\begin{proof}
Consider a circle of radius $R$ and area $A$ centered about the origin. By Gauss' circle problem \cite{Hardy}, the number of lattice points contained in this circle is \bea \label{gc1} \pi R^2 + O(R^{1+\epsilon}).\eea A generalization of Gauss' circle problem to count primitive lattice points -- that is, solutions to the equation $m^2 + n^2 \leq R^2$ with $m$ and $n$ coprime \cite{Wu} -- shows that the number of such solutions is
\bea \label{gc2}
\frac{6}{\pi} R^2 + O(R^{1 + \epsilon}).
\eea
Therefore, in the limit as $R$ goes to infinity, we find that the proportion of primitive lattice points tends to $6/{\pi^2}$.  By Lemma \ref{lem:greedyint} we know that if $a+bi \in H^*_{\Q(i),3}$ then $a+bi=k(c+di)$, where $c+di$ is a primitive lattice point and $k \in G_3^*$.

Suppose $a+bi \in H^*_{\Q(i),3}$ with $a^2+b^2<R^2$ and let $k=(a,b) \in G_3^*$. Then $a/k+(b/k)i$ is a primitive lattice point in a circle of radius $R/k$ and area $A/k^2$.  Since the proportion of such primitive lattice points tends to  $6/{\pi^2}$, each $k \in G_3^*$ will contribute a factor of $6/\pi^2k^2$ to the density of $H_{\Q(i),3}^*$. Thus
\begin{align} \label{eq:nathansum} d(H_{\Q(i),3}^*) \ =\  \frac{6}{\pi^2} \sum_{k \in G_3^*} \frac{1}{k^2} &\ = \ \frac{1}{\zeta(2)}\prod_p \sum_{i \in A_3^*} \frac{1}{p^{2i}}\nonumber\\
&\ = \ \ \prod_p \left({1-\frac{1}{p^2}}\right)\prod_{i=1}^\infty\left(1+\frac{1}{p^{2\cdot 3^i}}\right) \nonumber \\
&\ = \  \ \prod_p f(p^2)\ \approx\ 0.939735.
\end{align}
\end{proof}

This result is not unique to the Gaussian integers. For any quadratic number field $K$, let $H_{K,3}^*$ denote the greedy set of ideals of $\mathcal{O}_K$ which avoid progressions whose ratio is an ideal generated by a rational integer.  If we use the method of proof of Theorem \ref{greedyidealdensity}, we find that because the ideal generated by each rational prime $p$ has norm $N(p)=p^2$, the density of those ideals of $\mathcal{O}_K$ divisible exactly by an acceptable power of the ideal $(p)$ is $f(p^2)$.  Thus the Euler product for the density of $H_{K,3}^*$ behaves as it would in Theorem \ref{greedyidealdensity} if each prime $p$ remained inert.  Thus we have the following generalization of Theorem \ref{thm:densityintegralrat}:
\begin{theorem} Let $K$ be a quadratic number field.  The density of $H_{K,3}^*$, the greedy set of ideals of $\mathcal{O}_K$ which avoid progressions whose ratio is an ideal generated by a rational integer, is
\be\label{densityHk3}
d(H_{K,3}^*) \ =\  \prod_p f(p^2)\ \approx\ 0.939735.
\ee
\end{theorem}


\section{Bounds on the Upper Density}
\label{upperdensitysection}
We now turn our attention back to geometric progressions whose ratios lie in the quadratic number field, and consider the supremum of the upper densities of geometric-progression-free sets in this setting.

\subsection{Upper Bounds for the Upper Density}
In 1969, Riddell \cite{Riddell} gave the first upper bound for the upper density of a subset of the natural numbers which avoids 3-term geometric progressions. Here we generalize his arguments to give upper bounds for the upper density of a set of algebraic integers (or ideals) which avoid 3-term geometric progressions in a quadratic number field.

Riddell's argument gives an upper bound by excluding the fewest number of integers necessary to avoid progressions with ratios that are a power of 2. We utilize the same argument, replacing the prime 2 with whichever prime in the given number field has the smallest norm. We also need an estimate for how many ideals in our number field that have norm at most $x$. This is a well studied problem, going back to the work of Dedekind and Weber. It is known (see for example \cite[Theorem 5]{Murty}) that for a quadratic number field, $K$, the count, $N(x,K)$, of the number of ideals of $\mathcal{O}_K$ with norm at most $x$ satisfies \begin{equation}N(x,K) \ =\  \rho_k x +O\left(x^{1/2}\right),\label{linear}\end{equation}
where $\rho_K$ is a constant that depends on the number field.  In the specific case when $K = \Q(i)$, $\mathcal{O}_K=\Z[i]$ is the Gaussian integers and every ideal is principal. So $N(x,\Q(i))$ counts the Gaussian integers with norm up to $x$, which correspond to lattice points with distance less than $\sqrt{x}$ from the origin.  Thus, in this special case, \eqref{linear} follows from Gauss' circle problem, which tells us the number of integer lattice points contained in a circle centered at the origin with radius $r$ is $\pi r^2+O(r)$. (Note that the error term here can be improved but this suffices for our argument.)

In particular, when $K$ is any of the imaginary quadratic number fields with class number 1, equation \eqref{linear} tells us that the count of the algebraic integers with norm up to $x$ will grow linearly with $x$.

We can now see how the powers of the smallest prime element of a number field can be used to give an upper bound for the upper density of the set of algebraic integers (ideals) avoiding 3-term geometric progressions.  We consider first the representative case of the Gaussian integers.

\begin{theorem}\label{thm:riddelldensity}
The upper density of a subset of the Gaussian integers avoiding geometric progressions with non-unit ratios contained in $\Z[i]$ is at most $0.857143$.
\end{theorem}

\begin{proof}
We consider a subset of the Gaussian integers with norm at most $M$ and fix an element $r \in \Z[i]$ with smallest possible non-unit norm, $N\left(r\right)=2$.  (So $r$ is some associate of $1+i$.)  Then, if $b \in \Z[i]$ is any Gaussian integer coprime to $1+i$ with $N\left(b\right)\leq\frac{M}{4}$, we have that each element in the progression $\{b, br, br^2\}$ has norm at most $M$ and furthermore that distinct values of $b$ result in disjoint progressions.  Thus, we must exclude at least one Gaussian integer with norm at most $M$ for each such value of $b$.

We can then use Gauss' circle problem to count the number of such $b$, and so see that the proportion of Gaussian integers excluded in this way is
\be \label{TwoGPRatio}
\frac{1}{2} \cdot \frac{\#{b \text{ with } N(b)}\leq\frac{M}{4}}{\#{b \text{ with } N(b)}\leq M} \ =\  \frac{1}{2} \cdot\frac{\frac{M\pi}{4}+O\left(\sqrt{M}\right)}{M\pi+O\left(\sqrt{M}\right)}\ =\ \frac{1}{2} \cdot \frac{1}{2^2} +O\left(\frac{1}{\sqrt{M}}\right),
\ee
or 1/8 in the limit as $M \to \infty$.  Thus the upper density cannot be greater than 1/8.

This is a crude bound, but we can improve upon it by considering additional progressions.  Note that so long as $N(b)\leq \frac{M}{32}$ each element of a progression of the form $\{br^3, br^4, br^5\}$ will also have norm at most $M$, and furthermore these progressions will be disjoint not only from each other but also from the first set of progressions described above.  This results in an additional proportion of $(1/2)(1/2^5)$ Gaussian integers which must be excluded. Continuing this process yields an infinite series which gives the following upper bound:

\be\label{riddell}
1-\frac{1}{2^3}\sum_{n=0}^{\infty}\frac{1}{2^{3n}} \ = \ 1-\frac{\frac{1}{2^3}}{1-\frac{1}{2^3}} \ =\ \frac{6}{7}\ \approx\ 0.857143.
\ee
\end{proof}

We can extend this argument to give a corresponding upper bound for the upper density of a subset of algebraic integers in any class number 1 imaginary quadratic number field, or to the upper density of the ideals of any quadratic number field.  Using \eqref{linear}, we have that the number of algebraic integers (or ideals) with norm less than some given $N$ grows linearly with $N$.  We can therefore apply this same argument, replacing the common ratio, $r$, with an algebraic integer (or ideal respectively) with least possible non-unit norm.  By considering the possible behavior of the rational primes 2 and 3 in the quadratic number field we see that this least non-unit norm must be one of 2, 3, or 4.

In any quadratic number field containing an algebraic integer (ideal) of norm 2, we obtain the same upper bound as for the Gaussian integers, $6/7 \approx 0.857143$.  In the case where the smallest non-unit element has norm $3$, we have that the proportion of excluded elements is bounded above by
\be\label{eq:norm3}
1 - \frac{2}{3^3} \sum_{n=0}^\infty \frac{1}{3^{3n}} \ =\  1 - \frac{\frac{2}{3^3}}{1 - \frac{1}{3^3}} \ =\  \frac{12}{13}\ \approx\ 0.923077,
\ee
and in the case where the smallest non-unit element has norm $4$, we have that the proportion of excluded elements is bounded above by
\be\label{eq:norm4}
1 - \frac{3}{4^3} \sum_{n=0}^\infty \frac{1}{4^{3n}} \ =\  1 - \frac{\frac{3}{4^3}}{1 - \frac{1}{4^3}} \ =\  \frac{20}{21}\ \approx\ 0.952381,
\ee
both by arguments analogous to the case with norm $2$.

The results for the class number 1 quadratic fields are summarized in Table \ref{tab:riddellbounds}.

\begin{table}[h]
\centering
\begin{tabular}{|c|c|c|}
\hline
 $d$ & Upper Bound &Approximation \\
\hline
-1 & 6/7 & 0.857143 \\
-2 & 6/7 & 0.857143\\
-3 & 12/13 & 0.923077\\
-7 & 6/7 & 0.857143\\
-11 & 12/13 & 0.923077\\
-19 & 20/21 & 0.952381 \\
-43 & 20/21 & 0.952381 \\
-67 & 20/21 & 0.952381 \\
-163 & 20/21 & 0.952381\\
\hline
\end{tabular}
\vspace{.1em}
\caption{Upper bounds for the upper density of 3-term geometric-progression-free subsets of the algebraic integers in the class number 1 imaginary quadratic number fields, $\Q\left(\sqrt{d}\right)$.}
\label{tab:riddellbounds}
\end{table}


\subsection{Improvements on Upper Bounds for the Upper Density}

By generalizing Riddell's argument in our new setting, we were able to derive upper bounds for the upper density of sets of ideals containing no three terms in geometric progression. Notice, however, that this argument only takes into account progressions involving a ratio of smallest non-unit norm. We can conceivably improve these bounds by taking into account more possible ratios. We do so, generalizing an argument made over the integers in \cite{McNew}.

For a quadratic number field $ K = \Q\left(\sqrt{d}\right)$, and an integer $n \in \Z$, we define an element (or ideal) of the ring of integers $\mathcal{O}_K$ to be $n$-smooth if its factorization into irreducible elements consists entirely of elements with norm at most $ n$.

For the ring of integers in any class number 1 imaginary quadratic number field, we consider the $n$-smooth numbers up to a fixed norm and determine the exact number of integers which must be excluded from this set to ensure that it is progression-free. We derive the following improvement to Theorem \ref{thm:riddelldensity}.

\begin{theorem}
The upper density of a subset of the Gaussian integers avoiding geometric progressions with non-unit ratios contained in $\Z[i]$ is at most $0.85109$.
\end{theorem}

\begin{proof}
The 5-smooth elements of $\Z[i]$ consist of those elements whose only prime factors are  $1+i$, $2+i$, and $2-i$, that is, the three non-unit irreducible elements in this ring with smallest norm. Define the set $S_N$ to be the set of all elements of $\Z[i]$ (up to associates) with norm $\leq N$ such that their prime factorization consists only of these three smallest elements. That is, $S_N$ is the set of all 5-smooth numbers in $\Z[i]$ with norm $\leq N$. Considering $S_4 = \{1, 1+i, 2\}$ we see that these three elements form a progression, and therefore at least one element must be excluded from $S_4$ to ensure a set free of geometric progressions. Considering larger values of $N$, the next additional exclusion occurs when we hit $S_{20}=\{1,1+i,2,2+i,2-i,2+2i,1+3i,3+i,4,4+2i,4-2i\}$, at which point two additional exclusions must be made in order to obtain a set free of geometric progressions, bringing the total number of exclusions necessary up to three.

Now, corresponding to the set $S_4$, for any Gaussian integer $b$ such that $N(b) \leq M/4$ and such that $(b, 5 + 5i) = 1$ (where $5 + 5i = (1+i)(2+i)(2-i)$) we see that from the set $b\cdot S_4=\{b, (1+i)b, 2b\}$ we must necessarily exclude one element, which will have norm at most $M$. As in Riddell's argument, we see that for distinct values of $b$ the sets $b\cdot S_4$ are disjoint, and since the proportion of Gaussian integers, $b$, coprime to $5+5i$ is
\be \frac{(2-1)(5-1)(5-1)}{2\cdot 5\cdot 5}\ = \ \frac{16}{50}, \ee
we must exclude at least
\be \frac{16}{50}\left(\frac{M\pi}{4}\right) + O\left(\sqrt{M}\right)\ee of the $M\pi + O\left(\sqrt{M}\right)$ Gaussian integers with norm at most M, a density of $16/200=0.08$, which gives us an upper bound of 0.92 for the upper density of a geometric-progression-free set.

Similarly, corresponding to the set $S_{20}$, because we must exclude two additional elements corresponding to those $b$ where $N(b)\leq M/20$, we get an improved lower bound on the number of excluded elements of \be \frac{16M\pi}{50}\left(\frac{1}{4} + \frac{2}{20}\right) + O\left(\sqrt{M}\right),\ee
which in turn gives us the upper bound
\be 1-\frac{16}{50}\left(\frac{1}{4} + \frac{2}{20}\right)\ =\ 0.888. \ee
We can continue in this fashion, calculating each value of $N$ for which additional exclusion(s) are needed. The results for the Gaussian integers are summarized in table \ref{tab:exclusions}.

 \begin{table} [h]
 \centering
 \begin{tabular}{|c|c|}
 \hline
 $N$ & \# of integers excluded from $S_N$\\
 \hline
 $4$ & $1$\\
 $20$ & $3$\\
 $32$ & $4$\\
 $64$ & $5$\\
 $100$ & $8$\\
 $128$ & $9$\\
 $160$ & $11$\\
 $256$ & $12$\\
 $320$ & $14$\\
 $500$ & $18$\\
 \hline
 \end{tabular}
 \vspace{.1em}
 \caption{Norms up to 500 at which additional exclusions have to be made when avoiding 5-smooth geometric progressions in the Gaussian integers.}
 \label{tab:exclusions}
 \end{table}

 Using all these necessary exclusions, we must exclude
 \be \frac{16M\pi}{50} \left(\frac{1}{4} + \frac{2}{20} + \frac{1}{32} + \frac{1}{64} + \frac{3}{100} + \frac{1}{128} + \frac{2}{160} + \frac{1}{256} + \frac{2}{320} + \frac{4}{500}\right) + O\left(\sqrt{M}\right)\ee
 elements. So any geometric-progression-free-subset of the Gaussian integers with norms $\leq M$ can have size at most  $0.85109M\pi +O\left(\sqrt{M}\right)$.  Therefore we find an upper bound for the upper density of sets of Gaussian integers containing no three terms in geometric progression to be $0.85109$, which is an improvement over the previously computed upper bound of $0.857143$.
 \end{proof}

While this computation was done in $\Z[i]$, we can easily repeat this approach to find bounds for all other class number 1 imaginary quadratic number fields. In each case we have used the 3 elements of smallest non-unit norm and computed the values of $N$ at which additional exclusions were required.  The results of this computation are summarized in Table \ref{tab:nathanarg}.

 \begin{table}[h]
 \centering
 \begin{tabular}{|c|c|c|}
 \hline
  $d$ & Smallest Non-unit Norms & Upper Bound\\
 \hline
 -1 &2, 5, 5 & 0.851090\\
 -2 &2, 3, 3 & 0.839699\\
 -3 &3, 4, 7 & 0.910089\\
 -7 &2, 2, 7 & 0.858880\\
 -11 &3, 3, 4 & 0.917581\\
 -19 &4, 5, 5 & 0.949862\\
 -43 &4, 9, 11 & 0.945676\\
 -67 &4, 9, 17 & 0.946772\\
 -163 &4, 9, 25 & 0.946682\\
 \hline
 \end{tabular}
 \vspace{.1em}
 \caption{Improved upper bounds for the upper density of a set free of 3-term geometric progressions in each of the class number 1 imaginary quadratic number fields $\Q\left(\sqrt{d}\right)$.}
 \label{tab:nathanarg}
 \end{table}

\subsection{Lower Bounds for the Maximal Upper Density}

We next establish lower bounds on the supremum of the upper densities of subsets of imaginary quadratic number fields that avoid geometric progressions by generalizing an argument from \cite{McNew},  also discussed more recently by Nathanson and O'Bryant \cite{NatO'Bry2}. For illustration, we first consider the Gaussian integers and construct a subset $S \subset \Z[i]$ that has a large upper density and that avoids geometric progressions.

Fix a bound $M >0$ and note that if $x_1,x_2,$ and $x_3$ are Gaussian integers whose norms satisfy $M/4 < N(x_1) < N(x_2) < N(x_3) \leq M$ then $(x_1,x_2,x_3)$ cannot form a $3$-term geometric progression with non-unit ratio in $\Z[i]$.  This is because any non-unit ratio $r$ will have norm $N(r) \geq 2$, and hence $N(x_3)=N(x_1r^2)=N(x_1)N(r)^2 > M$. Thus, any three elements with norm in the interval $(M/4,M]$ do not comprise a geometric progression.  So by including all such elements in $S$, $S$ can include $3/4$ of the Gaussian integers with norm at most $M$.  In fact we can include more elements in $S$ fairly easily, utilizing the fact that the norm of any Gaussian integer lies in $\Z$.  We can also piece together sets of this sort for widely separated values of $M$ to produce a set avoiding these progressions with high upper density.
\begin{theorem}\label{thm:gaussianswiss}
We have that $0.844662$ is a lower bound for the supremum of the upper densities of all subsets of $\Z[i]$ which avoid 3-term geometric progressions.
\end{theorem}
\begin{proof}
We construct here a set $S$ which has the claimed upper density while avoiding 3-term geometric progressions over the Gaussian integers.  Let
\be\label{gaussianswissinterval}
T_{\hspace{-0.1mm}M} = \hspace{-0.3mm}\left(\hspace{-0.6mm}\frac{M}{6728},\frac{M}{6656} \right]\hspace{-0.2mm} \midcup \hspace{-0.1mm} \left(\hspace{-0.6mm}\frac{M}{3712},\frac{M}{3364} \right] \hspace{-0.2mm} \midcup \hspace{-0.1mm}\left(\hspace{-0.6mm}\frac{M}{3328},\frac{M}{928} \right] \hspace{-0.2mm} \midcup \hspace{-0.1mm} \left(\hspace{-0.6mm}\frac{M}{841},\frac{M}{832}\right] \hspace{-0.2mm} \midcup \hspace{-0.1mm} \left(\hspace{-0.6mm}\frac{M}{32},\frac{M}{8}\right] \hspace{-0.2mm} \midcup \hspace{-0.1mm} \left(\hspace{-0.6mm}\frac{M}{4},M\right]
\ee
and consider the set
\be
S_M \ = \ \{ x \in \Z[i] : N(x) \in T_M\}.
\ee
\\
We have already argued that the set $\{ x \in \Z[i] : N(x) \in (M/4,M] \}$ is free of 3-term geometric progressions, for a given $M$, using only the fact that the smallest non-unit ratio is 2. One can similarly check, using the fact that the complete list of the possible norms of Gaussian integers up to 83 is \begin{align*} &2 , 4 , 5 , 8 , 9 , 10 , 13 , 16 , 17 , 18 , 20 , 25 , 26 , 29 , 32 , 34 , 36 , 37 , 40 ,\\
&  41 , 45 , 49 , 50 , 52 , 53 , 58 , 61 , 64 , 65 , 68 , 72 , 73 , 74 , 80 , 81 , 82,\end{align*} that if any two terms of a geometric progression have norms contained in $T_M$ above, then the third element's norm must either be greater than $M$, less than $M/6728$ or lie in one of the omitted gaps.  We can visualize $S_M$ as the lattice points in a set of six annuli centered about the origin. Applying Gauss' circle problem, it follows that as $M \to \infty$ the proportion of Gaussian integers (up to norm $M$) with norm in $T_M$ is given by \be \frac{\left(\frac{M\pi}{1}{-}\frac{M\pi}{4}\right){+}\left(\frac{M\pi}{8}{-}\frac{M\pi}{32}\right){+}\left(\frac{M\pi}{832}{-}\frac{M\pi}{841}\right){+}\left(\frac{M\pi}{928}{-}\frac{M\pi}{3328}\right){+}\left(\frac{M\pi}{3364}{-}\frac{M\pi}{3712}\right){+}\left(\frac{M\pi}{6656}{-}\frac{M\pi}{6728}\right)}{M\pi} \ \approx\ 0.844662.\ee

We can now construct concentric, expanding sets of six annuli which still avoid geometric progressions in their union. Fix $M_0 = 1$ and recursively define $M_i = 6728 M_{i-1}^2$ for $i \geq 1$. Each $S_{M_i}$ forms a set of six annuli, with radii increasing as $i$ increases. Furthermore, these annuli are sufficiently separated in order to avoid any geometric progressions in their union $S := \bigcup_{i = 1}^{\infty} S_{M_i}$. To see this, suppose $x_1 \in S_{M_j}$ and $x_2=x_1r \in S_{M_k}$ for some $j \leq k$.  We've already argued that if $j=k$ then $x_3=x_1r^2 \notin S_{M_k}$.  Suppose then that $x_3 \in S_{M_l}$ for some $l>k$. Then \be N(r) \ = \ \frac{N(x_3)}{N(x_2)} \ > \  \frac{\frac{M_{k+1}}{6728}}{M_k}\ = \ M_k \ee
and so $N(x_2)= N(x_1r) \geq N(r) > M_k$.  But this contradicts $x_2 \in S_{M_k}$ and so we cannot have $x_3 \in S_{M_l}$ for any $l$.

Therefore, $S$ is free of 3-term geometric progressions by construction and has upper density $0.844662$.
\end{proof}

The arguments for the other imaginary quadratic number fields follow similarly. Care must be taken to ensure that the intervals $T_M$ are constructed to avoid introducing progressions with norms in the given imaginary quadratic number field. The results are summarized in Table \ref{tab:swisscheese}.

\begin{table}[h]
\begin{center}
    \begin{tabular}{|>{\centering\arraybackslash}p{1cm}|p{12cm}|>{\centering\arraybackslash}p{2cm}|}
      \hline
    $d$ & \centering Intervals Used &  Lower Bound \\ \hline
-1 & $\clopin{6728}{6656} \clopin{3712}{3364}\clopin{3328}{928}\clopin{841}{832}\clopin{32}{8} (M/4,M]$ & 0.844662\\ [.8cm]

-2 & $\clopin{19008}{16896}\clopin{8448}{2212}\clopin{48}{36}\clopin{32}{27}\clopin{24}{12}\clopin{9}{8} (M/4,M]$ & 0.818648 \\ [.8cm]

-3 & $\clopin{252}{63}\clopin{49}{36} (M/9,M]$ & 0.908163\\ [.2cm]

-7 & $\clopin{29696}{7424}\clopin{3712}{928}\clopin{32}{8}(M/4,M]$ & 0.844659\\
[.2cm]
-11 & $\clopin{405}{45} (M/9,M]$ & 0.908641\\
[.2cm]
-19 & $\clopin{2816}{176} (M/16,M]$ & 0.942826 \\
[.2cm]
-43 & $\clopin{1472}{1377}\clopin{576}{208}\clopin{81}{64} (M/16,M]$ & 0.943897 \\
 [.2cm]
-67 & $\clopin{1024}{729}\clopin{576}{144}\clopin{81}{64} (M/16,M]$ & 0.946382 \\
[.2cm]
-163 & $\clopin{2304}{2025}\clopin{1600}{1296}\clopin{1024}{729}\clopin{576}{144}\clopin{81}{64} (M/16,M]$ & 0.946589\\
\hline
    \end{tabular}
    \vspace{1mm}
    \caption{Lower bounds for the supremum of the upper densities of geometric-progression-free subsets of class number 1 imaginary quadratic number fields $\Q\left(\sqrt{d}\right)$.}
    \label{tab:swisscheese}
    \end{center}
\end{table}

\newpage

\section{Conclusions and Future Work}
In Section~\ref{greedysection}, we characterized the greedy set of ideals of the ring of integers over a quadratic number field avoiding 3-term geometric progressions, and thus obtained lower bounds for the supremum of the densities of geometric-progression-free sets. We believe, as in the case of the rational integers, that it is possible to construct sets with greater density by considering non-greedy sets.
\begin{que}
Is it possible to construct subsets of the integral ideals of a quadratic number field with density greater than that of the greedy set?  If so, how much greater?
\end{que}
\noindent We have also computed the densities of the greedy sets of ideals containing no 3-term geometric progressions over various quadratic number fields with small discriminant. We include here a plot of the distribution of densities of the greedy sets of ideals in all imaginary quadratic number fields, $\Q\left(\sqrt{-d}\right)$, with $1 \leq d \leq 10,000$. The histogram suggests some interesting structure, and we conjecture that it arises partially from the distribution of elements with small norms in a given number field. It would be interesting to examine the emergent patterns more rigorously.

\begin{figure}[h]
\centering
\includegraphics[scale=0.7]{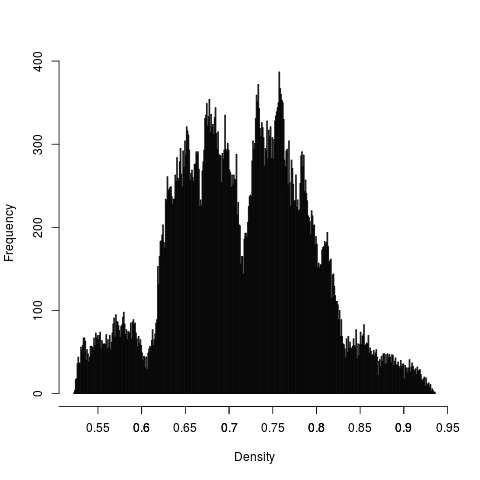}
\caption{Distribution of densities of greedy sets of ideals of imaginary quadratic number fields with discriminant at most 10,000.}
\end{figure}

In Section~\ref{upperdensitysection}, we generalized arguments by Riddell and by McNew to derive upper and lower bounds for the supremum of the upper densities of geometric-progression-free subsets of a quadratic number field which depend on the splitting behavior of small primes.

The questions we have answered in this paper leave us with many more potential avenues to explore. We end with the following open questions that we find interesting.

As mentioned, we restricted our study to geometric progressions whose ratios are non-unit algebraic integers. However, sets free of progressions with different ratios are also of interest and may require different methodologies to study.
\begin{que}
How would considering progressions with unit ratios change our results about the possible densities of geometric-progression-free subsets of the algebraic integers in an imaginary quadratic number field?
\end{que}

\begin{que}
How would these results change if one considered all progressions with non-unit ratios (not necessarily integral) in an imaginary quadratic number field, or fractional-ideal-ratios in an arbitrary quadratic number field?
\end{que}
We note that Rankin's set, $G_3^*$, also avoids all geometric progressions over the integers with rational ratio, although to this point we have only used the fact that it avoids integral-ratio geometric progressions.  In the same way, over a quadratic number field, $K$, one finds that the greedy set, $G_{K,3}^*$, constructed to avoid progressions with integral-ideal ratio, is also the greedy set avoiding progressions whose ratio is any fractional ideal of $K$.

\begin{que}
Is it possible to construct subsets of the ideals of the ring of integers, $\mathcal{O}_K$, of a quadratic number field $K$, avoiding progressions with fractional-ideal-ratio, with density greater than that of the greedy set $G_{K,3}^*$?
\end{que}

On the other hand, the arguments we used to give lower bounds for the supremum of the  possible upper densities would no longer apply in the fractional-ideal-ratio case, as they depend on the existence of a minimum norm for a ratio.

Moving beyond quadratic number fields, one can consider geometric-progression-free sets over number fields of higher degree or even over other rings, such as polynomials or matrices. Depending on the ring one considers, one may lose certain useful properties, such as commutativity, associativity, a nice lattice structure, or unique factorization.
\begin{que}
What results can be obtained about geometric-progression-free sets over a number field of arbitrary degree? Or over the quaternions? Or over the octonions?
\end{que}
\begin{que}
How do these results generalize to sets free of longer geometric progressions? 
\end{que} For other progressions of prime length, the situation should be very similar to the results presented here. However in the case of progressions of composite length the results may be more complicated, as the greedy set of integers avoiding arithmetic progressions of composite length is not so predictable as $A_3^*$ is.

\begin{que}
What can be said about the geometric-progression-free subsets of the algebraic integers (not the ideals) in an imaginary quadratic number field without unique factorization?  Can the greedy set be characterized? Can its density be determined?
\end{que}
\begin{que}
What issues are introduced when studying geometric-progression-free subsets of a matrix ring, such as $M_2(\Z)$,  which is no longer commutative?
\end{que}
\begin{que}
In a more general sense, how does the structure of a ring affect the possible sizes of its geometric-progression-free subsets?
\end{que}

\end{document}